\documentclass[12pt,letterpaper]{article}
\usepackage{url}

\usepackage{algpseudocode} 
\usepackage{algorithm}
\usepackage{amssymb}
\usepackage{graphicx}
\usepackage{fullpage}
\usepackage{amsfonts}
\usepackage{amsmath}
\usepackage{amsthm}
\usepackage{cite}

\begin{document}
\newtheorem{theorem}{Theorem}[section]
\newtheorem{lemma}[theorem]{Lemma}

\theoremstyle{definition}
\newtheorem{definition}[theorem]{Definition}
\newtheorem{example}[theorem]{Example}
\newtheorem{xca}[theorem]{Exercise}
\newtheorem{conjecture}[theorem]{Conjecture}

\theoremstyle{remark}
\newtheorem{remark}[theorem]{Remark}

\numberwithin{equation}{section}

\title{A computational upper bound on Jacobsthal's function}
\author{Fintan Costello,Paul Watts}

\date{}

\maketitle

\begin{abstract}
The function $h(k)$ represents the smallest number $m$ such that every sequence of $m$ consecutive integers contains an integer coprime to the first $k$ primes.  We give a new computational method for calculating strong upper bounds on $h(k)$.
\end{abstract}

\section{Introduction}

  Letting $p_i$ be the $i$th prime and $P_k$ be the product of the all primes up to $p_k$, Jacobsthal's function
   $h(k)$ represents the smallest number $m$ such that every sequence of $m$ consecutive integers contains an integer coprime to  $P_{k}$.  
The function $h(k)$ has been studied by a number of different authors, and is central to results on the maximal gaps between consecutive primes \cite{MaierPomerance}, \cite{Pintz} and on the least prime in arithmetic progressions \cite{Pomerance}. 
Explicit values of $h(k)$ are known only for $k \leq 49$, with the computation of $h(49)$ by Hagedorn taking $3$ months on a cluster of $30$ computers \cite{Hagedorn}.  

Let $g(n)$ represent the smallest number $m$ such that every sequence of $m$ consecutive integers contains an integer coprime to  $n$.  In a letter to Erdos \cite{ErdosJacobsthal}, Jacobsthal asked whether  
\begin{equation*}
\begin{split}
 h(k) &\leq C \ k^{2} \\
g(n) & \leq h(k) \textit{ for all $n$ with $k$ distinct prime factors}
\end{split}
\end{equation*} 
both hold for all positive $k$ and some constant $C$.  
Iwaniec's proof \cite{Iwaniec} that
\begin{equation*}
 h(k) \leq C \ (k \ log \ k )^{2}
\end{equation*} 
for an unknown constant $C$ gives our closest approach to the first of these conjectures.  
Hadju and Saradha  \cite{HadjuSaradhaJacobsthal} recently disproved the second conjecture using the explicit values of $h(k)$ calculated by Hagedorn.

 The best known explicit upper bounds on $h(k)$, of
\begin{equation*}
\begin{split}
h(k)&\leq 2^{k} \\ 
h(k) & \leq 2k^{2+2 e \ log \ k}
\end{split}
\end{equation*}
are due to Kanold \cite{Kanold}  and  Stevens \cite{Stevens}  respectively, with the second bound being stronger for $k \geq 260$ (these bounds in fact apply to $g(n)$ for all $n$ with $k$ distinct prime factors, and so cover the case where $n=P_k$).  These bounds are quite weak: while from Hagedorn's calculations we have  $h(49)=742$,   these bounds give $h(49) <  10^{15}$ and $h(49) <  10^{40}$ respectively.   We thus know relatively little about the explicit behavior of $h(k)$ for $k$ greater than $49$.    

In this paper we address this gap using a new computational method for computing explicit upper bounds on $h(k)$.  This method  gives bounds much stronger than those given by Kanold and Stevens; for example, this method gives a bound on $h(49)$ which is less than $3$ times the true value of $h(49)$.  This method is also much faster than that used by Hagedorn to calculate values of $h(k)$, computing a bound on $h(49)$ in seconds.   We used this method to compute upper bounds on $h(k)$ for  $k$ from $50$ to $10,000$: for all $k$ in that range we find
\begin{equation*}
  h(k) \leq 0.27749612254 \ k^2 \ log \ k 
\end{equation*} 
a bound hundreds of orders of magnitude stronger than those given by Kanold  and Stevens  in this range.

Our method is based on an expression for the function $\varphi(b,m,k)$,  which represents the smallest number $x$ such that every sequence of $m$ consecutive integers contains at least $x$ integers coprime to $P_k$.  Taking $\varphi(n)$ to be Euler's totient function and
 $\varphi_{min}(m,i)$ to be the minimum value of $\varphi(b,m,i)$ across all $b$,  we  prove
\begin{equation*}
\begin{split}
\varphi(P_k) \left\lfloor \frac{m}{P_k}\right\rfloor + r-   \sum\limits_{i=1}^{k}\left\lceil\frac{r}{p_i}\right\rceil&  + \sum\limits_{i=2}^{k}\left\lfloor\frac{r}{2 p_i}\right\rfloor +E   \\ &+  \sum\limits_{i=2}^{k-1}   \sum\limits_{j=i+1}^{k} 
\varphi_{min}\left(\left\lfloor \frac{r}{p_i p_j}\right\rfloor,i-1\right)  \leq \varphi_{min}(m,k)
\end{split}
\end{equation*}
for all $m$ and $k$, where $r =  m \textit{ mod } P_k$  and
where $E$ is a positive correction arising due to constraints on the co-occurence of residues of the primes up to $p_k$.
This expression has the computationally nice property that the recurrent double sum is dominated by terms where $i$ is small, and putting computational effort into calculating accurate  values of $\varphi_{min}( r ,i-1)$ across all candidate $r$'s for some set of low values of $i$  allows us to efficiently compute strong lower bounds on $\varphi(b, r ,i-1)$ for a large range of values of $k$. Since it is clear that
$$0 < \varphi_{min}(r ,i-1) \Rightarrow h(k) \leq m$$
this in turn  allows us to efficiently compute strong upper bounds on $h(k)$ across a large range of $k$.

In the first three parts of this paper we prove this expression.   In the last part we describe algorithms based on this expression that we use to compute explicit upper bounds on $h(k)$.

\section{Preliminaries} 

 We take $\omega(a)$ to represent the number of distinct prime factors of $a$,    $\omega_{k}(a)$ to be  the number of distinct primes that are factors of both $a$ and $P_k$, and $l_{k}(a)$ to be the lowest  factor of $a$ that is also a factor of $P_k$. 
 
For any set of integers $S$ we take  $F_{S}(d)$ to be the number of integers in $S$ that are divisible by $d$; for any integer $b$
we take $F_{b,m}(d)$ to be the number of integers in the sequence of $m$ consecutive integers $b+1, \ldots b+m$
that are divisible by $d$.

 We use the following result concerning divisors of members of arithmetic sequences.
\begin{theorem}
\label{sequenceTheorem}
 Let $d$ and $n$ be two coprime squarefree integers.  Then for any arithmetic sequence
\begin{equation*}
b+d,b+2d,\ldots,
\end{equation*}
there necessarily exists a corresponding sequence of consecutive integers
\begin{equation*}
c b+1,c b+2,\ldots
\end{equation*}
  such that
\begin{equation*}
GCD(c b+x,n)=GCD(b+xd,n) \textit{ for all $x$}
\end{equation*}
\end{theorem}
\begin{proof}
Choose integers $c$ and $z$ such that $c d - z n = 1$.  Since $c$ divides $= zn+1$ it is clear that $c$ and $n$ are coprime and so we have
\begin{equation*}
GCD(c(b+xd)-xzn, n) = GCD(b+xd,n) \textit{ for all $x$}
\end{equation*}
(because $xzn$ shares all of $n$'s prime factors while the only prime factors that $c(b+xd)$ shares with $n$ are those of $(b+xd)$).
Rearranging the left hand side (and using the fact that $c d - zn= 1$) we get
$$c(b+xd)-xzn = c b+x (c d - zn)=c b+x \textit{ for all $x$} $$
and the consecutive sequence starting at $cb+1$ has the required property.
\end{proof}

Following from this we define  $c_{b}(d)$ as follows:
\begin{definition}
For some integer $d$ with $p_i$ as its lowest prime factor, let $y+d$ be the first term in the sequence $b+1, \ldots b+m$ that has $d$ as a divisor.  We have then an arithmetic sequence
\begin{equation*}
y+d,y+2d,\ldots, y+F_{b,m}(d) \times d
\end{equation*}
all of whose members are divisible by $d$ and all of whom are in the sequence $b+1, \ldots b+m$ .  Then
$c_{b}(d)$ is an integer such that
\begin{equation*}
GCD(c_{b}(d)+x,P_{i-1})=GCD(y+xd,P_{i-1}) \textit{ for all $x$}
\end{equation*}
and so  the sequence of consecutive integers starting at $c_{b}(d)+1$ all have the same prime factors in common with $P_{i-1}$  as the corresponding terms in the arithmetic sequence of integers between $b+1$ and $b+m$ that have $d$ as a divisor. 
\end{definition}

\section{A recurrent expression for  $\varphi(b,m,k)$}

We now prove a recurrent expression for  $\varphi(b,m,k)$, the number of integers from $b+1$ to $b+m$ which are coprime to $P_{k}$.
\begin{theorem}
\label{omegaTheorem}
For  integers $m$, $b$ and $k$ we have
\begin{equation*}
\varphi(b,m,k) = m - \sum_{\substack{i=1}}^{k} F_{b,m}(p_{i})  + \sum_{\substack{b<a\leq b+m \\ \omega_{k}(a)>0}}(\omega_{k}(a)-1)
\end{equation*}
\end{theorem}
\begin{proof}
The expression
\begin{equation*}
 m - \sum_{\substack{i=1}}^{k} F_{b,m}(p_{i}) 
\end{equation*}
undercounts $\varphi(b,m,k)$ by $\omega_{k}(a)-1$ for each integer $a$ in our sequence that has $\omega_{k}(a)>0$, and so
\begin{equation*}
\varphi(b,m,k) = m-\sum_{\substack{i=1}}^{k} F_{b,m}(p_{i})    + \sum_{\substack{b<a\leq b+m \\ \omega_{k}(a)>0}}(\omega_{k}(a)-1)
\end{equation*}
as required.
\end{proof}

\begin{theorem}
\label{minFactorSum}
For any $p_x \mid P_k$ let $S$ be the set of integers  $a \in B$ that have $l_{k}(a)=p_x$.  Then
\begin{equation*}
 \sum_{\substack{a \in S}}(\omega_{k}(a)-1) =   \sum\limits_{\substack{x < i \leq k }} F_{S}(p_{i} p_{x})
\end{equation*}
\end{theorem}
\begin{proof}
Assume some $p_x \mid P_k$.  For any $a \in S$ the total number of composites $p_x p_i$ dividing $a$ where $p_i \mid P_k$ is therefore equal to the number of prime factors $p_i \neq p_x$ of $P_k$ which divide $a$.  Since $p_x \mid a$ this total is equal to $\omega_{k}(a)-1$. Each composite  $p_i p_x$ thus contributes $1$ to the sum
$$ \sum_{\substack{a \in S}}(\omega_{k}(a)-1)$$
for each $a \in S$ which has $p_i p_x$ as a divisor and so the total contribution that each such composite makes to that sum is $F_{S}(p_i p_x)$, and the result follows.
\end{proof}

\begin{theorem}
\label{recurrenceTheorem}
\begin{equation*}
 \sum_{\substack{b<a\leq b+m \\ \omega_{k}(a)>0}}(\omega_{k}(a)-1) = \sum\limits_{j=2}^{k}F_{b,m}(2p_j) +
 \sum\limits_{i=2}^{k-1}   \sum\limits_{j=i+1}^{k} \varphi(c_{b}(p_i p_j),F_{b,m}(p_{i} p_{j}),i-1)
\end{equation*}
\end{theorem}
\begin{proof}
The proof is inductive.  For the base of the induction, we note  from Theorem \ref{minFactorSum}  that
\begin{equation*}
\sum_{\substack{b<a\leq b+m \\ \omega_{1}(a)>0}}(\omega_{k}(a)-1) = \sum\limits_{j=2}^{k}F_{b,m}(2p_j) 
\end{equation*}

To prove induction we begin by assuming that
\begin{equation*}
\sum_{\substack{b<a\leq b+m \\ \omega_{x}(a)>0}}(\omega_{k}(a)-1) = \sum\limits_{j=2}^{k}F_{b,m}(2p_j) +
 \sum\limits_{i=2}^{x}  \sum\limits_{j=i+1}^{k} \varphi(c_{b}(p_i p_j),F_{b,m}(p_{i} p_{j}),i-1)
\end{equation*}
holds for some $x < k-1$. 

 Let $S$ be the set of integers in our sequence $b+1,\ldots, b+m$ which are coprime to $P_{x}$ and which have $p_{x+1}$ as a factor.  Since all integers not in $S$ have either already been counted or have $\omega_{x+1}(a)=0$, we have
 
$$\sum_{\substack{b<a\leq b+m \\ \omega_{x+1}(a)>0}}(\omega_{k}(a)-1) = \sum_{\substack{b<a\leq b+m \\ \omega_{x}(a)>0}}(\omega_{k}(a)-1) +\sum_{\substack{a \in S }}(\omega_{k}(a)-1)$$
 
By definition $l_{k}(a)=p_{x+1}$ for all $a \in S$, and so from Theorem \ref{minFactorSum}  we have
$$\sum_{\substack{a \in S }}(\omega_{k}(a)-1) = \sum\limits_{\substack{x+1 < j \leq k \\ }} F_{S}(p_j p_{x+1})$$

We can rewrite the right hand side here as
$$\sum\limits_{\substack{x+1 < i \leq k \\ }} F_{S}(p_i p_{x+1})=
\sum_{\substack{j=x+2}}^{k}  \sum_{\substack{b<a\leq b+m\\ p_{x+1} p_{j} \mid a \\ a \textit{ coprime to }P_{x}}} 1 $$

For each pair $p_{x+1} p_{j}$ we have an arithmetic subsequence consisting of those integers between $b+1$ and $b+m$ that have $p_{x+1} p_{j}$ as a factor.   This subsequence contains $F_{b,m}(p_j p_{x+1})$ integers. The right hand side in the above expression contributes $1$ for each member of this arithmetic subsequence which is coprime to $P_{x}$.  From Theorem \ref{sequenceTheorem} this arithmetic subsequence is equivalent to a sequence of $F_{b,m}(p_j p_{x+1})$ consecutive integers starting at $c_{b}(p_{x+1} p_j)$, and so the right hand side above contributes  $1$ for each member of this sequence which is coprime to $P_{x}$, giving a total contribution of
$$\varphi(c_{x+1,j},F_{b,m}(p_j p_{x+1}),x) $$
for each such pair.  We thus have
$$\sum_{\substack{a \in S }}(\omega_{k}(a)-1) = \sum\limits_{j=x+2}^{k} \varphi(c_{x+1,j},F_{b,m}(p_j p_{x+1}),x)$$
and so
\begin{equation*}
\sum_{\substack{b<a\leq b+m \\ \omega_{x+1}(a)>0}}(\omega_{k}(a)-1) = \sum\limits_{j=2}^{k}F_{b,m}(2 p_j) +
 \sum\limits_{i=2}^{x+1}  \sum\limits_{j=x+2}^{k} \varphi(c_{b}(p_i p_j),F_{b,m}(p_{i} p_{j}),i-1)
\end{equation*}
also holds for $x+1$.  This completes the induction  and gives the required result.
\end{proof}
Finally, combining  Theorems \ref{omegaTheorem} and \ref{recurrenceTheorem} we have
 \begin{theorem}
 \label{bound1}
 \begin{equation*}
\varphi(b,m,k) = m-\sum_{\substack{i=1}}^{k} F_{b,m}(p_{i})    + \sum\limits_{j=2}^{k}F_{b,m}(2 p_j) +  \sum\limits_{i=2}^{k-1}   \sum\limits_{j=i+1}^{k} \varphi(c_{b}(p_i p_j),F_{b,m}(p_{i} p_{j}),i-1)
\end{equation*}
for all $b,m$ and $k$.
 \end{theorem}

\section{A lower bound on $ \varphi_{min}(m,k) $  }
We now give a lower bound on $ \varphi_{min}(m,k)$, the minimum value of $ \varphi(b,m,k) $ across all $b$ .  This bound makes use of constraints on the co-occurence of residues of primes to $P_k$.
We begin with a very obvious result, which we give without proof.
\begin{theorem}
\label{ceiling}
If $d \nmid m$ then
$$F_{b,m}(d) = \left\lceil\frac{m}{d}\right\rceil \Leftrightarrow (b+m ) \ \textit{mod $d$}  < m  \ \textit{mod $d$}  
$$
\end{theorem}
Using this we get

\begin{theorem}
 \label{divisors1}
  For primes $p$ and $q$ let $ x = m \textit{ mod $ p$}$. If $x > 0$ and $q \mid m - x + p $ then 
  $$ F_{b,m}(p)= \left\lceil\frac{m}{p}\right\rceil \Rightarrow F_{b,m}(p q)= \left\lceil\frac{m}{p q}\right\rceil$$
\end{theorem}
\begin{proof}
Let $ x = m \textit{ mod $ p$}$.  Assume  $x > 0$, $q \mid m - x + p  $ and
 $$F_{b,m}(p)= \left\lceil\frac{m}{p}\right\rceil$$  Then since $x>0$ means $p \nmid m$ from Theorem \ref{ceiling}
we have $$(b+m) \textit{ mod $ p$} < x$$
Let $$y = b+m - ( (b+m) \textit{ mod $ p$})$$
and $p$ divides every integer in the arithmetic sequence $y, y-p, \ldots ,y-(q - 1)p$.  Since there are $q$ terms in this sequence, one of these terms is also divided by $q$, and so  
$$(b+m) \textit{ mod $ p q$} \leq b+m - y-(q - 1)p$$
However $$b+m -y =  (b+m) \textit{ mod $ p$} < x $$
and so
$$(b+m) \textit{ mod $ p q$} < x+(q - 1)p$$
By assumption we have $q \mid m - x + p  $;  by definition $p \mid m-x$ and so we have $p q \mid m - x + p  $ and $x <p <p q$. We thus get  $m \textit{ mod $p q$} =  x - p + p q=x + (q-1)p$ and so
$$(b+m) \textit{ mod $ p q$} < m \textit{ mod $p q$}$$
and  from Theorem \ref{ceiling} we get the required result.
 \end{proof}

  \begin{theorem}
 \label{divisors2}
  For any integer $m$ there exists an integer $b$ such that $\varphi(b,m,k)=\varphi_{min}(m,k)$ and such that for all odd primes $p \mid m-1$ we have $$ F_{b,m}(p)= \left\lceil\frac{m}{p}\right\rceil \Rightarrow F_{b,m}(2 p)= \left\lceil\frac{m}{2 p}\right\rceil  $$
\end{theorem}
\begin{proof}
Let $b$ be some integer such that $\varphi(b,m,k)=\varphi_{min}(m,k)$ and $ 2 \mid b+m$.  (To see that we will always be able to select  such a $b$, note that for any $a$ such that $\varphi(a,m,k)=\varphi_{min}(m,k)$ we either have $ 2 \mid a+m$  or $ 2 \nmid a+m$.  If  $ 2 \mid a+m$ then $b=a$ satisfies our requirements, whereas if $ 2 \nmid a+m$ then $ 2 \mid a+m+1$ and so $\varphi(a+1,m,k)=\varphi(a,m,k)=\varphi_{min}(m,k)$ and $b=a+1$ satisfies our requirements.) For any $p \mid m-1$ we have
$$ F_{b,m}(p)= \left\lceil\frac{m}{p}\right\rceil \Rightarrow(b+m ) \ \textit{mod $p$}  < m  \ \textit{mod $p$} = 1$$
(from Theorem \ref{ceiling}) and so for any $p \mid m-1$ we have
$$ F_{b,m}(p)= \left\lceil\frac{m}{p}\right\rceil \Rightarrow p \mid (b+m ) \Rightarrow 2p \mid (b+m ) \Rightarrow  F_{b,m}(2p)= \left\lceil\frac{m}{2p}\right\rceil $$
as required.   
\end{proof}

  Combining these results we get the following
\begin{theorem}
\label{headerBound}
For integer $m$ let $r =  m \textit{ mod } P_k$  
and let
 \begin{equation*}
E= |\left\{  i:   1<i\leq k, p_i \nmid r, 2 \mid (r - (r \textit{ mod $ p_i$})+p_i) \textit{ or } p_i \mid r-1 \right\} |  
\end{equation*}
Then
\begin{equation*}
\begin{split}
\varphi(P_k) \left\lfloor \frac{m}{P_k}\right\rfloor + r-   \sum\limits_{i=1}^{k}\left\lceil\frac{r}{p_i}\right\rceil&  + \sum\limits_{i=2}^{k}\left\lfloor\frac{r}{2 p_i}\right\rfloor +E   \\ &+  \sum\limits_{i=2}^{k-1}   \sum\limits_{j=i+1}^{k} 
\varphi_{min}\left(\left\lfloor \frac{r}{p_i p_j}\right\rfloor,i-1\right)  \leq \varphi_{min}(m,k)
\end{split}
\end{equation*}
for all $m$ and $k$.
\end{theorem} 
\begin{proof}
Assume  $r$ and $E$ as defined above.  Let $b$ be an integer for which the conditions in Theorem \ref{divisors2} hold.  Since Euler's totient $\varphi(n)$ gives the number of integers coprime to $n$ in any sequence of $n$ consecutive integers we have
$$ \varphi(b,m,k) = \varphi(P_k) \left\lfloor \frac{m}{P_k}\right\rfloor + \varphi(b,r,k)$$  and we need only consider  the value of $\varphi(b,r,k)$.  
From Theorems \ref{divisors1} and \ref{divisors2} we see that 
$$F_{b,r}(p)= \left\lceil\frac{r}{p}\right\rceil \textit{ and } F_{b,r}(2 p)= \left\lfloor\frac{r}{2p}\right\rfloor$$
cannot hold for any prime  counted in the definition of $E$, and so we have 
\begin{equation*}
r-   \sum\limits_{i=1}^{k}\left\lceil\frac{r}{p_i}\right\rceil + \sum\limits_{i=2}^{k}\left\lfloor\frac{r}{2 p_i}\right\rfloor +E    
\leq r-\sum\limits_{i=1}^{k} F_{b,r}(p_{i})   + \sum\limits_{j=2}^{k}F_{b,r}(2 p_i)  
\end{equation*}

Combining this with the fact that by definition
$$ \varphi_{min}\left(\left\lfloor\frac{m}{p_{i} p_{j}}\right\rfloor,i-1\right) \leq  \varphi(c_{b}(p_i p_j),F_{b,m}(p_{i} p_{j}),i-1)$$
we get 
\begin{equation*}
\begin{split}
\varphi(P_k) \left\lfloor \frac{m}{P_k}\right\rfloor & + r-   \sum\limits_{i=1}^{k}\left\lceil\frac{r}{p_i}\right\rceil   + \sum\limits_{i=2}^{k}\left\lfloor\frac{r}{2 p_i}\right\rfloor +E   +  \sum\limits_{i=2}^{k-1}   \sum\limits_{j=i+1}^{k} 
\varphi_{min}\left(\left\lfloor \frac{r}{p_i p_j}\right\rfloor,i-1\right) \\ \leq 
& m-\sum_{\substack{i=1}}^{k} F_{b,m}(p_{i})    + \sum\limits_{j=2}^{k}F_{b,m}(2 p_j)  +  \sum\limits_{i=2}^{k-1}   \sum\limits_{j=i+1}^{k} \varphi(c_{b}(p_i p_j),F_{b,m}(p_{i} p_{j}),i-1)
\end{split}
\end{equation*}
From Theorem \ref{bound1} the right hand side of this expression is equal to $\varphi(b,m,k)$, which from Theorem \ref{divisors2} is equal to $\varphi_{min}(m,k)$ and  we get the required result.
 \end{proof}

\section{Computations }

We first describe Algorithm $1$, which  computes a recursive function $\varphi_{low}(m,k)$.  Values of this function give a lower bound on $\varphi_{min}(m,k)$ as given in Theorem \ref{headerBound}.

\newcommand{\OneLineIf}[2]{ \State \algorithmicif\ {#1} \  \algorithmicthen \ {#2} \algorithmicend \ \algorithmicif }
\newcommand{\OneLineIfElse}[3]{ \State \algorithmicif\ {#1} \  \algorithmicthen \ {#2} \algorithmicelse \ {#3} \algorithmicend \ \algorithmicif }
\newcommand{\OneLineFor}[2]{ \State \algorithmicfor\ {#1} \  \algorithmicdo \ {#2} \algorithmicend \ \algorithmicfor }

\begin{algorithm}
\label{alg1}
\begin{algorithmic}[4] 
\Function{$\varphi_{low}$}{m,k}
\State $S \gets \left\lfloor\frac{m}{P_{k}}\right\rfloor$
\State $L \gets \varphi(P_{k}) S$
\State $r \gets m - S$
\OneLineIf{$k \leq 6$}{ \Return  $L   +  v\left( r,k \right)$}
\OneLineIf{  $r < 2 p_{k-1}$}{ \Return L }
\OneLineIf{$UseKnown = True$ and  $k \leq 49$  and  $r <h(k)$ }{ \Return L }
\State $L \gets L + r  - \left\lceil\frac{r}{2} \right\rceil $
\OneLineFor{$i = 2 $ to $k$}{$ L \gets  L - \left\lceil\frac{r}{p_i} \right\rceil + \left\lfloor\frac{r}{2p_i} \right\rfloor$}
\State $E \gets 0$
\For{$i = 2 $ to $k$}
	 \If{  $p_i \nmid r$ and ($ 2 \mid (r- (r\textit{ mod } p_i)+p_i)$ or $p_i \mid (r-1)$)} 
	\State $ E \gets E + 1$ 
	 \EndIf
\EndFor 
 \State $i \gets 2$ 
\Repeat
     \State  $j=i+1$ 
     \Repeat 
      \State $m_{new} \gets \left\lfloor \frac{r}{p_i p_j}\right\rfloor$
      \State $U \gets   \varphi_{low}\left( m_{new},i-1 \right)$
      \OneLineIf{$U > 0$}{$ L  \gets L +  U$}
  \State $ j \gets j+1$
  \Until{$j=k $ or $U \leq 0$}
  \State $ i \gets i+1$
\Until{$i =k$} 
\OneLineIfElse{$L>0$}{ \Return L }{\Return 0}
\EndFunction
\end{algorithmic}
\caption{The function $\varphi_{low}(m,k)$.  This function requires that explicit values
of $\varphi_{min}(m, k)$ are known for all $k$ up to $6$ and all $m \leq P_{k}$. Known values of $h(k)$
for $k$ less than $50$ are used only if the variable $UseKnown$ is set.}
\end{algorithm}

In this algorithm the variables $L$ is used to hold the incrementally computed lower bound on $\varphi_{min}(m,k)$.
Lines $2 $ to $4$ obtain a value for $\left\lfloor \frac{m}{P_k} \right\rfloor$, assign an initial value for $L$ as in Theorem \ref{headerBound}, and obtain a value for $r$ (again as in Theorem \ref{headerBound}).  
  
  We explicitly computed the value of $\varphi_{min}(m,i)$ for all $i$ less than or equal to $6$ and for all $m$ less than $P_i$; these provide stopping conditions for the recursion in $\varphi_{low}(m,k)$  (line $5$).  
  
From Hagedorn \cite{Hagedorn} we  have explicit values of $h(k)$ for $k \leq 49$, and we also have the general result that $ 2 p_{k-1} \leq h(k)$ for all $k$.  Since $\varphi_{min}( m ,k)=0$ if $m < h(k)$ we use both the explicit values and the general result as further stopping conditions for the recursion (lines $6$ and $7$).  Hagedorn's explicit values for $h(k)$ for $k \leq 49$ are used only if the variable $UseKnown$ is set to true.

Lines $8$ and $9$  calculate the first three terms in the left hand side of Theorem \ref{headerBound}, with lines $9$ to $14$ calculating the $E$ term in that sum. 

Lines $16$ to $26$ give two nested loops containing recursive calls to $\varphi_{low}(\cdot)$ and so calculating  a lower bound on value of the double sum in Theorem \ref{headerBound}.  For each pair of values for indices $i$ and $j$ in these loops, the algorithm recursively gets a lower bound on  $\varphi_{min}(m_{new},i-1)$ as in Theorem \ref{headerBound}, placing this bound in a variable $U$.

Finally, line $27$ returns a lower bound on $\varphi_{min}(m, k)$. Since $\varphi_{min}(m, k)$ cannot
be less than $0$ (no sequence contains a negative number of primes) the algorithm
returns the computed bound $L$ only if $L > 0$; otherwise $0$ is returned.

At lines $24$ and $26$ the inner and outer loops terminate at $j=k$ and $i=k-1$ respectively, as in the double sum in Theorem \ref{headerBound}.  For efficiency the inner loop also terminates if the last obtained recursive lower bound value $U$ was equal to $0$: in this case all further recursive calls within this loop will also return a value of $0$, because $m_{new}$ is decreasing on each cycle of the inner loop.  Similarly, the outer loop  terminates if the first recursive bound value $U$ obtained on the previous cycle of the inner loop had a value of $0$: in this case recursive calls in  all further cycles of the inner loop will also return a value of $0$, because all values of $m_{new}$ in those further cycles will be less than the current value of $m_{new}$.

We implemented the function $\varphi_{low}$ in the computer algebra system PARI \cite{PARI}.  We take $b(k)$ to represent the lowest integer $m$ such that $\varphi_{low}(m,k)>0$ and so $b(k)$ is an upper bound on $h(k)$.  We find $b(k)$ using linear search across increasing values of $m$  (Algorithm $2$).

\begin{algorithm}
\label{alg2}
\begin{algorithmic}  
\State $m \gets initialM$ 
\For{$k = start$ to $end$}
	 \While{  $\varphi_{low}(m,k) < 1$}
 \State $m \gets m+1$ 
 \EndWhile 
 \State print(k,m)
\EndFor
\end{algorithmic}
\caption{Calculating $b(k)$ for  $k$ from $start$ to $end$.}
\end{algorithm}

To compare values of $b(k)$ with the values of $h(k)$ calculated by Hagedorn for $k$ from $1$ to $49$  \cite{Hagedorn}  we ran Algorithm $2$ with $start=1$, $end=49$, $initialM=1$ and variable $UseKnown$  set to $False$.   In this range the bound $b(k)$ was less than $3$ times the true value of $h(k)$. Figure $1$ graphs $h(k)$, $b(k)$, and  Steven's and Kanold's bounds on $h(k)$ in this range.

To calculate values of $b(k)$ up to $k=1000$ we ran Algorithm $2$ with $start=50$, $end=1000$,  variable $UseKnown$  set to $True$, and $initialM$ set to $h(49)=742$.

   For a given $k$ Algorithm $2$ calculates values of $\varphi_{low}(m,k)$ for increasing values of $m$ starting at $m=b(k-1)+1$ and continuing until $b(k)$ is reached (that is, until $\varphi_{low}(m,k) > 0$).  Clearly the number of integers $m$ for which Algorithm $2$ must calculate $\varphi_{low}(m,k)$ grows as $k$ grows.  Algorithm $2$ will thus be relatively slow for large $k$ (in our runs Algorithm $2$ took around 15 minutes to calculate $b(1000)$).  For this reason  when calculating bounds $b(k)$ for $k$ greater than $1000$ we took an alternate approach,  using the function  $\varphi_{low}$ to find the lowest integer $C_k$ such that $$\varphi_{low}\left(\left\lfloor \frac{C_k \ k^2}{10000}\right\rfloor,k\right)>0$$ (see Algorithm $3$).  For each $C_k$ we then have $h(k) \leq (C_k \ k^{2})/10000$.  To calculate values of $b(k)$ up to $k=10000$ we ran Algorithm $3$ with $start=1001$, $end=10000$,  variable $UseKnown$  set to $True$, and $initialC$ set to $10000$.  

 \begin{algorithm}
\label{alg3}
\begin{algorithmic}  
\State $C_k \gets initialC$ 
\For{$k = start$ to $end$}
\State $S \gets C_k$ 
	 \While{  $\varphi_{low}\left(\left\lfloor \frac{C_k \ k^{2}}{10000}\right\rfloor,k\right) >0$}
 \State $C_k \gets C_k-1$
 \EndWhile 
 \If{$C_k < S$}
     \State $C_k \gets C_k+1$
 \Else
 	 \While{  $\varphi_{low}\left(\left\lfloor \frac{C_k \ k^{2}}{10000}\right\rfloor,k\right) <1$}
 \State $C_k \gets C_k+1$
 \EndWhile  
 \EndIf
 \State $print\left(k,\left\lfloor \frac{C_k \ k^{2}}{10000}\right\rfloor\right)$
\EndFor
\end{algorithmic}
\caption{Calculating an upper bound on $b(k)$ for $k$ from $start$ to $end$.}
\end{algorithm}

\pagebreak

 Figure $1$  graphs $b(k)$ for $k$ from $50$ to $10000$ as obtained from these  algorithms.
We find
\begin{equation*}
  b(k) \leq 0.27749612254 \ k^2 \ log \ k 
\end{equation*} 
for all $k$ in this range, and so this  gives an upper bound on $h(k)$ for these $k$.

Figure $2$ compares $h(k)$, $b(k)$, Kanold's bound and Steven's bound for $k$ up to $49$.
Figure $3$ compares the logs of $b(k)$, Stevens' bound, and Kanold's bounds for $k$ up to $10000$. From this graph we see that $b(k)$ is hundreds of orders of magnitude stronger than Stevens' and Kanold's bounds in this range.


\bibliographystyle{amsplain}

\bibliography{references}     
 \begin{figure}[t]
 \label{BkTo10000}
\scalebox{0.6}[0.6]{
      \includegraphics[viewport= 0 500  640 720]{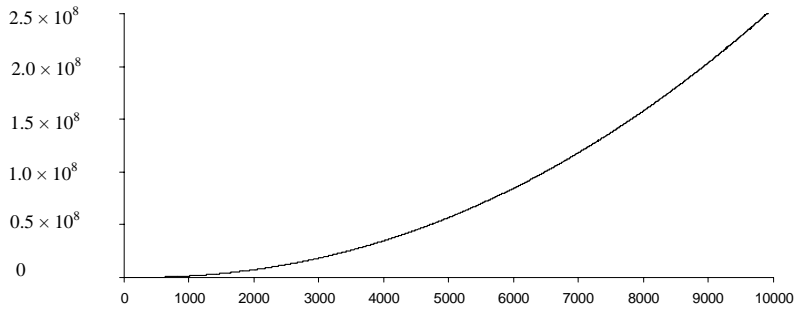}
     } 
\caption{Graph of $b(k)$ vs $k$ for $k$ from $50$ to $10000$}
\end{figure}

 \begin{figure}[t]
 \label{compareTo49}
\scalebox{0.65}[0.65]{
      \includegraphics[viewport= 50 520  680 740]{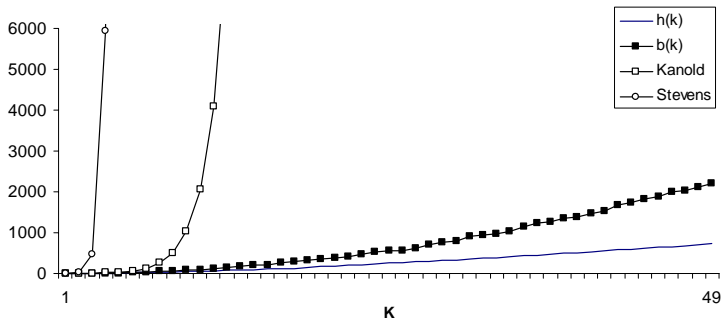}
     } 
\caption{Graph comparing $h(k)$, $b(k)$, Kanold's bound and Steven's bound for $k$ up to $49$.}
\end{figure}

  \begin{figure}[t]
 \label{compareLogs}
\scalebox{0.65}[0.65]{
      \includegraphics[viewport= 50 520  680 740]{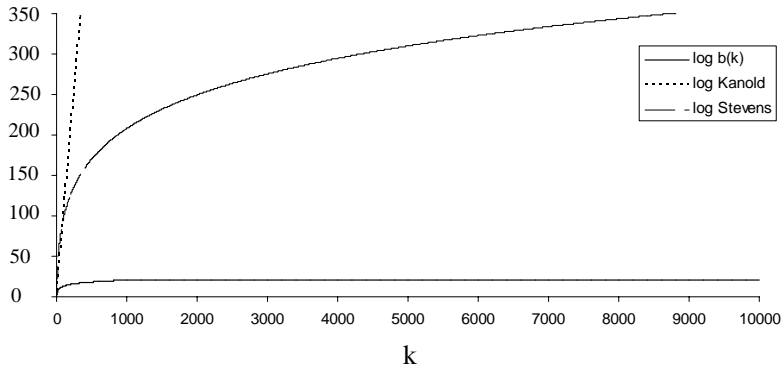}
     } 
\caption{Graph comparing  the log of $b(k)$ with the logs of Kanold's and Steven's bounds for $k$ to $10000$.}
\end{figure}

\end{document}